\documentclass[12pt,twoside,reqno]{amsart}
\usepackage{amssymb, latexsym}

\newcommand{\beqa}{\begin{eqnarray*}}
\newcommand{\eeqa}{\end{eqnarray*}}
\newcommand{\beqn}{\begin{eqnarray}}
\newcommand{\eeqn}{\end{eqnarray}}

\newcommand{\iy}{\infty}

\newcommand{\B}{\mathbb B}
\newcommand{\C}{\mathbb C}

\newcommand{\R}{\mathbb R}

\newcommand{\N}{\mathbb N}

\newcommand{\K}{\mathbb K}

\newcommand{\mcE}{\mathcal E}

\newcommand{\mcL}{\mathcal L}

\newcommand{\mfB}{\mathfrak B}

\newcommand{\mfM}{\mathfrak M}

\newcommand{\la}{\mu}

\newcommand{\m}{\mu}

\newcounter{cnt1}
\newcounter{cnt2}
\newcounter{cnt3}
\newcommand{\blr}{\begin{list}{$($\roman{cnt1}$)$}
 {\usecounter{cnt1} \setlength{\topsep}{0pt}
 \setlength{\itemsep}{0pt}}}
\newcommand{\bla}{\begin{list}{$($\alph{cnt2}$)$}
 {\usecounter{cnt2} \setlength{\topsep}{0pt}
 \setlength{\itemsep}{0pt}}}
\newcommand{\bln}{\begin{list}{$($\arabic{cnt3}$)$}
 {\usecounter{cnt3} \setlength{\topsep}{0pt}
 \setlength{\itemsep}{0pt}}}
\newcommand{\el}{\end{list}}
\newtheorem{thm}{Theorem}[section]
\newtheorem{lem}[thm]{Lemma}
\newtheorem{cor}[thm]{Corollary}

\newtheorem{Def}[thm]{Definition}
\newtheorem{prop}[thm]{Proposition}
\newtheorem{rem}[thm]{Remark}
\newcommand{\Rem}{\begin{rem} \rm}
\newcommand{\bdfn}{\begin{Def} \rm}
\newcommand{\edfn}{\end{Def}}

\newcommand{\ba}{\begin{array}}
\newcommand{\ea}{\end{array}}

\begin{document}
\begin{center}\large{{\bf{Kuelbs-Steadman spaces on Separable Banach spaces}}} 
 
  Hemanta Kalita$^{1}$ , Bipan Hazarika$^{2}$, Timothy Myers$^3$\\
$^{1}$Department of Mathematics, Gauhati University, Guwahati 781014, assam, India\\
$^{2}$Department of Mathematics, Gauhati University, Guwahati 781014, assam, India\\
 $^{3}$Department of Mathematics, Howard University, Washington DC
20059

Email:  hemanta30kalita@gmail.com;
 bh\_rgu@yahoo.co.in; bh\_gu@gauhati.ac.in ; timyers@howard.edu

\end{center}
\title{}
\author{}
\thanks{} 

\begin{abstract} The purpose of this paper is to construct a new class of separable Banach spaces  $\K^p[\mathbb{B}], \; 1\leq p \leq \infty$.  Each of these spaces contain the $ \mcL^p[\mathbb{B}] $ spaces, as well as the space $\mfM[\R^\iy]$, of finitely additive measures as dense continuous compact embeddings. These spaces are of interest because they  also contain the Henstock-Kurzweil integrable functions on $\mathbb{B}$. 
Finally, we offer a interesting approach to the Fourier transform on $\K^p[\mathbb{B}].$  
{\footnotesize {\bf{AMS subject classification \textrm{(2020)}:}}} 26A39, 46B03, 46B20, 46B25.
\end{abstract}
\maketitle


\pagestyle{myheadings}
\markboth{\rightline {\scriptsize   Kalita, Hazarika,Myers}}
        {\leftline{\scriptsize  }}

\maketitle
\section{Introduction and Preliminaries}
     
       T.L. Gill and T. Myers \cite{GM} introduced  a new theory of Lebesgue measure on $\R^\infty;$ the construction of which is virtually the same as the development of Lebesgue measure on $\R^n.$ This theory can be useful in formulating a new class of spaces which will provide a Banach Spaces which will provide a Banach Space structure for Henstock-Kurzweil (HK) integrable functions. This later integral is interest because it generalizes the Lebesgue, Bochner and Pettis integrals see for instance \cite{RA,RH,HT,KZ,SY,BST}. As a major drawback of HK-integrable function space is not naturally Banach space (see \cite{AL,RA,HS,HT,H,KZ,KW,CS} references therein).  In \cite{TY}, Yeong mentioned if canonical approach can developed  then the above drawback can be solved. Gill and Zachary \cite{TG},  introduced a new class of Banach spaces $K{S^p}[\Omega],~ \forall~ 1\leq p \leq \infty $ (Kuelbs-Steadman spaces) and $\Omega \subset \mathbb{R}^n $ which are canonical spaces (also see \cite{KB}). These spaces are separable that contains the corresponding $ \mcL^p $ spaces as dense, continuous, compact embedding. There main aim was to find these spaces contains the Denjoy integrable functions as well as additive measures. They found that these spaces are perfect for highly oscillatory functions that occur in quantum theory and non linear analysis.\\
       Throughout our paper, we suppose the notation $\mathbb{B}_j^\infty$ and assume that $J=[-\frac{1}{2}, \frac{1}{2}]$ is understood. We denote $\mathcal{L}^1,~\mathcal{L}^p $ are classical Lebesgue spaces.   
        Our study focused on the main class of Banach spaces $\K^p[\mathbb{B}_j^\infty],~1\leq p \leq \infty.$  These spaces contains  
        the HK-integrable functions and contains the $\mcL^p[\mathbb{B}_j^\infty] $ spaces, $1\leq p \leq \infty $ as continuous dense and compact embedding.  
         \begin{Def}\cite{CS}
   A function $ f:[a,b] \to \mathbb{R} $ is Henstock integrable if there exists a function $ F:[a,b] \to \mathbb{R} $ and for every $ \epsilon > 0 $ there is a function $ \delta(t) > 0 $ such that for any $ \delta-$fine partition $ D=\{[u,v], t \} $ of $I_0=[a,b], $ we have $$|| \sum[f(t)(v-u)-F(u,v)]|| < \epsilon,$$ where the sum $ \sum $ is  run over $ D= \{ ([u,v], t) \}$ and $F(u,v)= F(v)-F(u).$     We write $ H\int_{I_0} f=F(I_0).$
   \end{Def}
 \begin{Def}\cite{GM} 
   Let $ \mathbb{A}_n = \mathbb{A} \times J_n$ and $ \mathbb{B}_n = \mathbb{B} \times J_n $ ($n^{th}$ order box sets in $ \mathbb{R}^\infty$). We define 
   \begin{enumerate}
   \item $ \mathbb{A}_n \cup \mathbb{B}_n =(\mathbb{A} \cup \mathbb{B}) \times J_n;$
   \item $ \mathbb{A}_n \cap \mathbb{B}_n = (\mathbb{A} \cap \mathbb{B} ) \times J_n; $
   \item $ \mathbb{B}_{n}^{c} = \mathbb{B}^c \times J_n.$
   \end{enumerate}
   \end{Def}

 \begin{Def}
  \cite{TG} We define $\mathbb{R}_{J}^{n} =\mathbb{R}^n \times J_n.$ If $ T $ is a linear transformation on $ \mathbb{R}^n $ and $ \mathbb{A}_n = \mathbb{A} \times J_n ,$ then $ T_I $ on $ \R_J^n $ is defined by $ T_J[\mathbb{A}_n] = T[\mathbb{A}] $. We define $ B[\R_J^n]$ to be the Borel $\sigma-$algebra for $ \R_J^n,$ where the topology on $ \R_J^n $ is defined via the class of open sets $ D_n = \{ U \times J_n : U  \mbox{~is~ open~ in~} \mathbb{R}_J^n\}.$ For any $ a \in B[\mathbb{R}^n], $ we define $\mu_{\mathbb{B}}(\mathbb{A}_n) $ on $ \R_J^n $ by product measure $ \mu_{\mathbb{B}}(\mathbb{A}_n) = \mu\mathbb{A}_n(\mathbb{A}) \times \Pi_{i= n +1 }^{\infty} \mu_J(J) = \mu\mathbb{A}_n(\mathbb{A}).$
  \end{Def}
  Clearly 
 $ \mu_{\mathbb{B}}(.) $ is a measure on $ B[\R_J^n] $ is equivalent to $n$-dimensional Lebesgue measure on $ \mathbb{R}_J^n.$
  The measure $ \mu_{\mathbb{B}}(.) $ is both translationally and rotationally invariant on $ (\R_J^n, B[\R_J^n]) $ for each $ n \in \mathbb{N}.$
   Recalling the  theory on $ \R_J^n $ that completely paralleis that on $ \mathbb{R}^n. $     Since $ \R_J^n \subset \mathbb{R}_{J}^{n+1},$ we have an increasing sequence, so we define $ \widehat{\mathbb{R}}_{J}^{\infty} = \lim\limits_{ n \to \infty} \R_J^n = \bigcup\limits_{k=1}^{\infty} \mathbb{R}_{J}^{k}.$ Let $X_1= \widehat{\R}_J^\infty $ and let $\tau_1$ be the topology induced by the class of open sets $D \subset X_1$ such that $D= \bigcup\limits_{n=1}^{\infty}D_n = \bigcup\limits_{n=1}^{\infty}\{U \times J_n : U$ is open in $\R^n\}$. Let $X_2= \R^\infty \setminus\widehat{\R}_J^\infty$ and let $\tau_2$ be discrete topology on $X_2$ induced by the discrete metric so that , for $x,y \in X_2,~x \neq y,~d_2(x,y)=1$ and for $x=y~d_2(x,y)=0$
   \begin{Def}
   \cite{TG} We define $(\mathbb{B}_j^\infty, \tau) $ be the co-product $(X_1, \tau_1) \otimes (X_2, \tau_2)$ of $(X_1, \tau_1) $ and $(X_2, \tau_2)$ so that every open set in $(\mathbb{B}_j^\infty, \tau)$ is the disjoint union of two open sets $G_1 \cup G_2$ with $G_1 $ in $(X_1, \tau_1)$ and $G_2 $ in $(X_2, \tau_2)$. 
   \end{Def}
   It follows that $\mathbb{B}_j^\infty = \R^\infty$ as sets. However, since every point in $X_2$ is open and closed in $\mathbb{B}_j^\infty$ and no point is open and closed in $\R^\infty,$ So, $\mathbb{B}_j^\infty \neq \R^\infty$ as a topological spaces.  In \cite{GM} shown that it  can extend the measure $ \mu_{\mathbb{B}}(.) $ to $ \mathbb{R}^\infty $.    
 
Similarly, if $B[\R _J^n ]$ is the Borel $\sigma$-algebra for $\R _J^n ,$ then $B[\R_J^n ] \subset B[\R _J^{n+1}] $by: $$\widehat{B}[\mathbb{B}_j^\infty]  = \lim\limits_{n \to \infty} B[\R_J^n ] =\bigcup\limits_{k=1}^{\infty}B[R_J^k].$$Let $B[\mathbb{B}_j^\infty] $be the smallest $\sigma$-algebra containing $\widehat{B}[\mathbb{B}_j^\infty] \cup P(\R^\infty \setminus \bigcup\limits_{k=1}^{\infty}[\R_J^k]),$ where $P(.)$ is the power set.
It is obvious that the class $B[\R _J^\infty ]$ coincides with the Borel $\sigma-$algebra generated by the $\tau-$topology on $\R_J ^\infty.$
 \subsection{Measurable function}
  We discuss about measurable function on $\mathbb{B}_j^\infty$ as follows.  
     Let $ x =(x_1,x_2,\dots) \in \mathbb{B}_j^\infty,$  $ J_n = \Pi_{k=n+1}^{\infty}[\frac{-1}{2}, \frac{1}{2}] $ and let $ h_n(\widehat{x})= \chi_{J_n}(\widehat{x}),$ where $\widehat{x} = (x_i)_{i=n+1}^{\infty}.$
     \begin{Def}
     \cite{TG} Let $M^n $ represented the class of measurable functions on $\R^n.$ If $ x \in \mathbb{B}_j^\infty$ and $f^n \in M^n,$ Let $\overline{x}= (x_i)_{i=1}^{n} $ and define an essentially tame measurable function of order $n( $ or $ e_n-$ tame ) on $\mathbb{B}_j^\infty$ by $$f(x)= f^n(\overline{x}) \otimes h_n(\widehat{x}).$$ We let $M_J^n = \{f(x)~: f(x) = f^n( \overline{x}) \otimes h_n(\widehat{x}),~x \in \mathbb{B}_j^\infty \}$ be the class of all $e_n-$ tame function.
     \end{Def}
      
     \begin{Def}
     A function $f: \mathbb{B}_j^\infty \to \R$ is said to be measurable and we write $ f \in M_J$, if there is a sequence $\{f_n \in M_J^n \} $ of $e_n-$ tame functions, such that $$\lim\limits_{ n \to \infty} f_n(x) \to f(x)~\mu_{\mathbb{B}}-(a.e.).$$
     \end{Def}
     This definition highlights  our requirement that all functions on infinite dimensional space must be constructively defined as (essentially) finite dimensional limits.
     The existence of functions satisfying above definition is not obvious. So, we have 
     \begin{thm}
     (Existence) Suppose that $ f: \mathbb{B}_j^\infty \to (- \infty, \infty)$ and $f^{-1}(a) \in B[\mathbb{B}_j^\infty]$ for all $ a \in B[\R]$ then there exists a family of functions $\{f_n \},~f_n \in M_J^n $ such that $f_n(x) \to f(x) , \mu_{\mathbb{B}}(-a.e.)$
     \end{thm}
     \begin{rem}
     Recalling that any set $a,$ of non zero measure is concentrated in $X_1$ that is $\mu_{\mathbb{B}}(a)= \mu_{\mathbb{B}}(a \cap X_1) $ also follows that the essential support of the limit function $f(x)$ in definition $1.9$ i.e. $\{x:~f(x) \neq 0 \}$ is concentrated in $\R_J^n,$ for some $N.$
     \end{rem}
     \subsection{Integration theory on $\mathbb{B}_j^\infty$}
      We discuss a constructive theory of integration on $\mathbb{B}_j^\infty$ using the known properties of integration on $\R_J^n.$ This approach has the advantages that all the theorems for Lebesgue measure apply. Proofs are similar as for the proof on $\R^n.$ Let $\mcL^1[\R_J^n]$ be the class of integrable functions on $\R_J^n,$ Since $\mcL^1[\R_J^n] \subset \mcL^1[\R_I^{n+1}]$ we define $\mcL^1[\widehat{\R}_J^\infty]= \cup_{n=1}^{\infty}\mcL^1[\R_J^n].$
      \begin{Def}
      We say that a measurable function $ f \in \mcL^1[\mathbb{B}_j^\infty],$ if there is a Cauchy-sequence $\{f_n \} \subset \mcL^1[\widehat{\R}_J^\infty]$ with $f_n \in \mcL^1[\R_J^n] $ and $$\lim\limits_{n \to \infty}f_n(x) = f(x),~\mu_{\mathbb{B}}-(a.e.)$$
      \end{Def}
      \begin{thm}
      $\mcL^1[\mathbb{B}_j^\infty] = \mcL^1[\widehat{\R}_J^\infty]$
      \end{thm}
      \begin{proof}
      We know that $ \mcL^1[\R_J^n] \subset \mcL^1[\widehat{R}_J^\infty]$ for all $n.$ we sufficiently need to prove $\mcL^1[\widehat{R}_J^\infty]$ is closed. Let $ f $ be a limit point of $\mcL^1[\widehat{R}_J^\infty]~( f \in \mcL^1[\mathbb{B}_j^\infty])$. If $f=0$ then the result is proved. So we consider $f \neq 0.$ If $a_f$ is the support of $f,$ then $\mu_{\mathbb{B}}(a_f)= \mu_{\mathbb{B}}(a_f \cap X_1).$ Thus $a_f \cup X_1 \subset \R_J^n $ for some $N.$ This means that there is a function $g \in \mcL^1[\R_J^{N+1}] $ with $ \mu_{\mathbb{B}}(\{x:~f(x) \neq g(x) \})=0$. So, $f(x)=g(x)-$a.e. as $\mcL^1[\R_J^n]$ is a set of equivalence classes.So,  $\mcL^1[\mathbb{B}_j^\infty] = \mcL^1[\widehat{\R}_J^\infty]$
      \end{proof}
      \begin{Def}
      If $f \in \mcL^1[\mathbb{B}_j^\infty],$ we define the integral of $f$ by $$ \int_{\mathbb{B}_j^\infty} f(x) d \mu_{\mathbb{B}}(x)= \lim\limits_{n \to \infty} \int_{\mathbb{B}_j^\infty} f_n(x) d \mu_{\mathbb{B}}(x)$$ where $\{f_n \} \subset \mcL^1[\mathbb{B}_j^\infty] $ is any cauchy-sequence converging to $f(x)-$a.e.
      \end{Def}
      \begin{thm}
      If $f \in \mcL^1[\mathbb{B}_j^\infty] $ then the above integral exists and all theorems that are true for $f \in \mcL^1[\R_J^n],$ also hold for $ f \in \mcL^1[\mathbb{B}_j^\infty].$
      \end{thm}
     
     \subsubsection{Class of $\mathbb{B}_J^\infty,$ where $\B$ is a separable Banach space:}          
     As an application of $\mathbb{B}_j^\infty$, we can construct $\mathbb{B}_J^\infty,$ where $\B$ is separable Banach space.  Recalling  a sequence $(e_n) \in \B $ is called a Schauder basis (S-basis) for $\B$, if $||e_n||_\B=1 $ and for each $x \in \B,$ there is a unique sequence $(x_n) $ of scalars such that $$x=\lim\limits_{k \to \infty}\sum_{n=1}^{k}x_ne_n = \sum_{n=1}^{\infty}x_ne_n.$$
    any sequence $(x_n) $ of scalars associated with a $ x \in \B, ~\lim\limits_{n \to \infty}x_n=0$. Let $$j_k=\left[\frac{-1}{2\ln(k+1)},\frac{1}{2\ln(k+1)} \right] $$ and $$j^n=\Pi_{k=n+1}^{\infty}j_k,~j= \Pi_{k=1}^{\infty}J_k.$$ Let $\{e_k\} $ be an S-basis for $\B,$ and let $x=\sum\limits_{n=1}^{\infty}x_n e_n,$ and from $\mathcal{P}_n(x)= \sum\limits_{k=1}^{n}x_ke_k $ and $\mathcal{Q}_nx=(x_1, x_2,\dots,x_n) $ then  $\mathbb{B}_j^n$ is defined by $$\mathbb{B}_j^n=\{\mathcal{Q}_n x : x \in \B\} \times j^n $$ with norm $$||(x_k)||_{\mathbb{B}_j^n}= \max_{1 \leq k \leq n}||\sum\limits_{i=1}^{k}x_ie_i||_{\B}=\max_{1 \leq k \leq n} ||\mathcal{P}(x)||_\B.$$
  as $\mathbb{B}_j^n \subset \mathbb{B}_J^{n+1} $ so we can set $\mathbb{B}_J^\infty = \bigcup\limits_{n=1}^{\infty}\mathbb{B}_j^n $ and $\mathbb{B}_J $ is a subset of $\mathbb{B}_J^\infty$. \\
  We set $\mathbb{B}_J $ as $\mathbb{B}_J= \left\{(x_1, x_2,\dots): \sum\limits_{k=1}^{\infty}x_ke_k \in \B\right\}$ and norm on $\mathbb{B}_J $ by $$||x||_{\mathbb{B}_J}= \sup_{n}||\mathcal{P}_n(x)||_\B=||x||_\B.$$
   If we consider $\B[\mathbb{B}_J^\infty] $ be the smallest $\sigma$-algebra containing $\mathbb{B}_J^\infty$ and define $\B[\mathbb{B}_J]= \B[\mathbb{B}_J^\infty] \cap \mathbb{B}_J$ then by a known result 
   \begin{equation}\label{eq1}
   ||x||_\B= \sup_{n}||\sum_{k=1}^{n}x_ke_k||_\B.
   \end{equation}
   is an equivalent norm on $\B$. 
   \begin{prop}\cite{TG}
     When $\B $ carries the equivalent norm (\ref{eq1}), the operator $$T:(\B, ||.||_\B) \to (\mathbb{B}_J, ||.||_{\mathbb{B}_J}) $$ defined by $T(x)=(x_k) $ is an isometric isomorphism from $\B$ onto $\mathbb{B}_J.$
    \end{prop}
    This shows that every Banach space with an S-basis has a natural embedding in $\mathbb{B}_j^\infty.$\\
     So we call $\mathbb{B}_J $ the canonical representation of $\B $ in $\mathbb{B}_j^\infty.$ 
     With $\B[\mathbb{B}_J]= \mathbb{B}_J \cap \B[\mathbb{B}_J^\infty] $ we define $\sigma-$algebra generated by $\B$ and associated with $\B[\mathbb{B}_J] $ by $$\mathbb{B}_J[\B]=\{T^{-1}(a)~| a \in \B[\mathbb{B}_J]\} = T^{-1}\{\B[\mathbb{B}_J]\}.$$ Since $\mu_{\mathbb{B}}(a_j^n)=0 $ for $a_j^n \in \B[\mathbb{B}_j^n]$ with $a_j^n$ compact, we see $\mu_{\mathbb{B}}(\mathbb{B}_j^n)=0,~n \in \N$. So, $\mu_{\mathbb{B}}(\mathbb{B}_J)=0 $ for every Banach space with an S-basis. Thus the restriction of $\mu_{\mathbb{B}} $ to $\mathbb{B}_J $ will not induce a non trivial measure on $\B.$ 
     \begin{Def}
     \cite{TG, YY}
      We define $\overline{v}_k,~\overline{\m}_k $ on $a \in B[\R]$ by $$\overline{v}_k(a)= \frac{\mu(a)}{\mu(j_k)},~\overline{\m}_k(a)= \frac{\mu(a \cap j_k)}{\mu(j_k)}$$ and for an elementary set $a = \pi_{k=1}^{\infty}a_k \in B[\mathbb{B}_j^n],~$ define $\overline{V}_j^n $ by $$\overline{V}_j^n= \pi_{k=1}^{n}\overline{v}_k(a) \times \pi_{k={n+1}}^{\infty}\overline{\m}_k(a).$$  
     \end{Def}
     We define $V_j^n $ to be the Lebesgue extension of $\overline{V}_j^n $ to all $\mathbb{B}_j^n$ and $V_j(a)= \lim\limits_{n \to \infty}V_j^n(a)~\forall a \in B[\mathbb{B}_J].$ We adopt a variation of method developed by Yamasaki \cite{YY}, to define $V_j^n $ to the Lebesgue extension of $\overline{V}_j^{n} $ to all of $\mathbb{B}_j^n$ and define $V_j(\B)= \lim\limits_{n \to \infty}V_j^n(\B), \forall \B \in \B[\mathbb{B}_J].$
       
     \begin{rem}              
     $\mathbb{B}_J $ is image of $\B $ in $\mathbb{B}_j^\infty$. The image becomes an isometric isomorphism if the norm we define $||u_1, u_2,\dots,u_n||_{\mathbb{B}_J}=||u||_\B.$      
          \end{rem}
          
    \subsubsection{Integration Theory on $\mathbb{B}_J^\infty:$}
        In this section , we discuss integration on a separable Banach space $\mathbb{B}_J^\infty$ with an $S-$basis. Even if a reasonable theory of Lebesgue measure exists on $\mathbb{B}_J^\infty,$ this is not sufficient to make it a useful mathematical tool. In additional, all the theory developed for finite dimensional analysis. One must be able to approximate an infinite-dimensional problem as a natural limit of the finite dimensional case in a manner that leads itself to computational implementation. Recalling $\mu_{\mathbb{B}}$ restricted to $B[\mathbb{B}_{j}^{n}] $ is equivalent to $\mu\mathbb{A}_n.$ We required that the integral restricted to $B[\mathbb{B}_j^n]$ to the integral on $\R^n.$ Let $f: \mathbb{B} \to [0, \infty]$ be a measurable function and let $\mu_{\mathbb{B}}$ be constructed using the family $\{j_k\}$. If $\{i_n\} \subset M $ is an increasing family of non negative simple functions with $i_n \in M_J^n,$ for each $n$ and $\lim\limits_{n \to \infty}i_n(x) = f(x),~\mu_{\mathbb{B}}-$a.e. We define the integral of $f$ over $\mathbb{B}_J^\infty$ by $$\int_{\mathbb{B}_J^\infty}f(x)d\mu_{\mathbb{B}} = \lim\limits_{n \to \infty}\int_{\mathbb{B}_J^\infty}[i_n(x) \pi_{i=1}^{n}\mu(J_i)]d \mu_{\mathbb{B}}(x).$$
        Let $\mcL^1[\mathbb{B}_j^n] $ be the class of integrable functions on $\mathbb{B}_j^n.$ Since $\mcL^1[\mathbb{B}_j^n] \subset \mcL^1[\B_{j}^{n+1} ], $ we define $\mcL^1[\widehat{\B}_j^n] = \bigcup_{n=1}^{\infty}\mcL^1[\mathbb{B}_j^n].$
        \begin{enumerate}
        \item We say that a measurable function $ f \in \mcL^1[\mathbb{B}_J^\infty] $ if there exists a Cauchy sequence $\{f_m\} \subset \mcL^1[\widehat{\B}_j^\infty],$ such that $$\lim\limits_{m \to \infty}\int_{\mathbb{B}_J^\infty}|f_m(x)- f(x)|d \mu_{\mathbb{B}}(x)=0.$$ That is a measurable function $ f \in \mcL^1[\mathbb{B}_J^\infty] $ if there exists a Cauchy sequence $\{f_m\} \subset \mcL^1[\widehat{\B}_j^\infty],$ with $f_m \in \mcL^1[\mathbb{B}_j^n] $ and $$\lim\limits_{m \to \infty}f_m(x) = f(x),~\mu_{\mathbb{B}}-(a.e.)$$
        \item We say that a measurable function $f \in C_0[\mathbb{B}_J^\infty],$ the space of continuous functions that vanish at infinity, if there exists a Cauchy sequence $\{f_m\} \subset C_0[\widehat{\B}_j^\infty],$ such that $$\lim\limits_{m \to \infty}\int_{\mathbb{B}_J^\infty}\sup_{ x \in \mathbb{B}_J^\infty}|f_m(x)- f(x)|d \mu_{\mathbb{B}}(x)=0.$$
        \end{enumerate}
        \begin{thm}
        $\mcL^1[\widehat{\B}_j^\infty] = \mcL^1[\mathbb{B}_J^\infty].$
        \end{thm}
        \begin{Def}
        If $f \in \mcL^1[\mathbb{B}_J^\infty],$ we define the integral of $f$ by  $$\lim\limits_{m \to \infty}\int_{\mathbb{B}_J^\infty}f_m(x)d \mu_{\mathbb{B}}(x) = \int_{\mathbb{B}_J^\infty}f(x)d \mu_{\mathbb{B}}(x),~\mu_{\mathbb{B}}-(a.e.)$$ where $\{f_m\} \subset \mcL^1[\mathbb{B}_J^\infty] $ is any Cauchy sequence converging to $f(x)-$a.e.
        \end{Def}
       \begin{thm}
        If $f \in L ^1 [\mathbb{B}_J^\infty],$ then the above integral exists and all theorems that are true for $f \in \mcL^1[\mathbb{B}_j^n],$ also hold for $f \in \mcL^1[\mathbb{B}_J^\infty].$
        \end{thm}
       \begin{lem}
\cite{KB}
( Kuelbs Lemma): If $\B$ is a separable Banach space, there exists a separable Hilbert space $H \supset \B $ as a continuous dense embedding.
\end{lem}

    \section{The Kuelbs-Steadman  space $\K^p[\mathbb{B}]$, where $\mathbb{B}$ is  separable Banach space:}
  Now we proceed for the construction of the canonical space $\K^p[\mathbb{B}_J^\infty].$ 
    Let $\mathbb{B}_j^n$ be separable Banach space with $S-$basis, let $\K^p[\mathbb{\widehat{B}}_j^n] = \bigcup\limits_{k=1}^{\infty}\K^p[\mathbb{B}_j^k],$ and let $C_0[\mathbb{\widehat{B}}_j^n] =\bigcup\limits_{n=1}^{\infty}C_0[\mathbb{B}_j^n].$
    \begin{Def}
    A measurable function $ f \in \K^p[\mathbb{B}_j^n] $ if there exists a Cauchy sequence $\{f_m\} \subset \K^p[\widehat{\mathbb{B}}_j^n],$ with $f_m \in \K^p[\widehat{\mathbb{B}}_j^n] $ and $$\lim\limits_{m \to \infty}f_m(x) = f(x),~\mu_{\mathbb{B}}-(a.e.)$$
     \end{Def}   
        \begin{thm}
        $\K^p[\widehat{\B}_j^n] = \K^p[\mathbb{B}_j^n].$
        \end{thm}
        \begin{Def}
        If $f \in \K^p[\mathbb{B}_j^n],$ we define the integral of $f$ by  $$\lim\limits_{m \to \infty}\int_{\mathbb{B}_j^n}f_m(x)d \mu_{\mathbb{B}}(x) = \int_{\B}f(x)d \mu_{\mathbb{B}}(x),~\mu_{\mathbb{B}}-(a.e.)$$ where $\{f_m\} \subset \K^p[\mathbb{B}_j^n] $ is any Cauchy sequence converging to $f(x)-$a.e.
        \end{Def}
       \begin{thm}
        If $f \in \K^p [\mathbb{B}_j^n],$ then the above integral exists.
        \end{thm}   
\subsection{\bf{For the construction of   $\K^p[\mathbb{B}_j^n]$ :}}We start with $\mcL^1[\mathbb{B}_j^n],$ Choosing a countable dense set of sequences $\{\mcE_n(x)\}_{n=1}^{\infty}$ on the unit ball of $\mcL^1[\mathbb{B}_j^n] $ and assume $\{\mcE_n^*\}_{n=1}^{\infty} $ be any corresponding set of duality mapping in $\mcL^\infty[\B]$, also if $\B $ is $\mcL^1[\mathbb{B}_j^n]$ , using Kuelbs lemma, it is clear that the Hilbert space $\K^2[\mathbb{B}_j^n]$ will contain some non absolute integrable functions. For fix $t_k>0 $ such that $ \sum_{k=1}^{\infty} t_k=1$ and defined an inner product $(.)$ on $\mcL^1[\mathbb{B}_j^n]$ by $$(f,g)= \sum_{k=1}^{\infty}t_k[ \int_{\mathbb{B}_j^n}\mcE_k(x) f(x)d \mu_{\mathbb{B}}(x)][\int_{\mathbb{B}_j^n}\mcE_k(y)g(y) d \mu_{\mathbb{B}}(y)]^c$$ The completion of $\mcL^1[\mathbb{B}_j^n]$ in  the  inner product is the  spaces  $\K^2[\mathbb{B}_j^n].$  We can  see directly that $\K^2[\mathbb{B}_j^n] $ contains the HK-integrable functions.
   We call the completion of $\mcL^1[\mathbb{B}_j^n]$ with the above inner product, the Kuelbs-Steadman space $\K^2[\mathbb{B}_j^n].$
    \begin{thm}\label{th62}
     The space $\K^2[\mathbb{B}_j^n]$ contains $\mcL^p[\mathbb{B}_j^n]$ (for~each $p,~1\leq p <\infty$) as dense subspace.
    \end{thm}
    \begin{proof}
     We know $\K^2[\mathbb{B}_j^n] $ contains $\mcL^1[\mathbb{B}_j^n]$ densely. Thus we need to only show $L^q[\mathbb{B}_j^n] \subset \K^2[\mathbb{B}_j^n] $ for $ q \neq 1.$   Let $ f \in L^q[\mathbb{B}_j^n]$ and $q < \infty.$\\
     Since $ |\mcE(x)|=\mcE(x) \leq 1 $ and $|\mcE(x)|^q \leq \mcE(x),$ we have 
     \begin{align*}
     ||f||_{\K^2} &= \left[\sum\limits_{n=1}^{\infty}t_k\left|\int\limits_{\mathbb{B}_j^n}\mcE_k(x)f(x)d\la_{\mathbb{B}}(x)\right|^\frac{2q}{q}\right]^\frac{1}{2}\\&\leq \left[\sum\limits_{n=1}^{\infty}t_k\left(\int\limits_{\mathbb{B}_j^n}\mcE_k(x)|f(x)|^qd\la_{\B}(x)\right)^\frac{2}{q}\right]^\frac{1}{2}\\&\leq\sup\limits_{k}\left(\int\limits_{\mathbb{B}_j^n}\mcE_k(x)|f(x)|^qd\la_{\B}(x)\right)^\frac{1}{q} \leq ||f||_q.
     \end{align*}
     Hence $f \in \K^2[\mathbb{B}_j^n].$
    \end{proof}
    We can construct  the norm of $\K^p[\mathbb{B}_j^n]$  which is  defined as  $$||f||_{\K^p[\mathbb{B}_j^n]} =\left\{\begin{array}{c}\left(\sum\limits_{k=1}^{\infty}t_k\left|\int_{\mathbb{B}_j^n}\mathcal{E}_k(x)f(x)d\mu_{\mathcal{B}}(x)\right|^p\right)^{\frac{1}{p}}, \mbox{~for~} 1\leq p<\infty;\\
  \sup\limits_{k\geq 1}\left|\int_{\mathbb{B}_j^n}\mathcal{E}_k(x)f(x)d\mu_{\mathcal{B}}(x)\right|, \mbox{~for~} p=\infty \end{array}\right.$$
   
  It is easy to see that $||.||_{\K^p[\mathbb{B}_j^n]} $ defines a norm on $\mcL^p[\mathbb{B}_j^n].$ If $\K^p[\mathbb{B}] $ is the completion of $\mcL^p[\B]$ with respect to this norm, we have 
  \begin{thm}
  For each $q,~1\leq q < \infty,~L^q[\mathbb{B}_j^n] \subset \K^p[\mathbb{B}_j^n]$ as dense continuous embedding.
  \end{thm}
  \begin{proof}
   We know  from previous Theorem \ref{th62}, and by the construction  $\K^p[\mathbb{B}_j^n] $ contains $\mcL^p[\mathbb{B}_j^n]$ densely. Thus we need only show $L^q[\mathbb{B}_j^n] \subset \K^p[\mathbb{B}_j^n] $ for $ q \neq p.$ 
   Let $ f \in L^q[\mathbb{B}_j^n]$ and $q < \infty.$ 
   Since $|\mcE(x)|=\mcE(x) \leq 1 $ and $|\mcE(x)|^q \leq \mcE(x),$ we have 
   \begin{align*}
   ||f||_{\K^p} &= [\sum\limits_{n=1}^{\infty}t_k|\int\limits_{\mathbb{B}_j^n}\mcE_k(x)f(x)d\la_{\B}(x)|^\frac{qp}{q}]^\frac{1}{p}\\&\leq [\sum\limits_{n=1}^{\infty}t_k(\int\limits_{\mathbb{B}_j^n}\mcE_k(x)|f(x)|^qd\la_{\B}(x))^\frac{p}{q}]^\frac{1}{p}\\&\leq\sup\limits_{k}(\int\limits_{\B}\mcE_k(x)|f(x)|^qd\la_{\B}(x))^\frac{1}{q}\\&\leq ||f||_q.
   \end{align*}
   Hence $f \in \K^p[\mathbb{B}_j^n].$
  \end{proof}
  \begin{cor}
   $\mcL^\infty[\mathbb{B}_j^n] \subset \K^p[\mathbb{B}_j^n].$
  \end{cor}
  \begin{thm}
  $\C_c[\mathbb{B}_j^n] $ is dense in $\K^2[\mathbb{B}_j^n].$
  \end{thm}
  \begin{proof}
   As $\C_c[\mathbb{B}_j^n] $ is dense in $\mcL^p[\mathbb{B}_j^n]$    and $\mcL^p[\mathbb{B}_j^n] $ densely contained in $\K^2[\mathbb{B}_j^n].$ 
   Hence the proof.
  \end{proof}
  \begin{rem}
As H\"{o}lder and generalized H\"{o}lder inequalities for $\mcL^p[\mathbb{B}_j^n]$ is hold for $ 1 \leq p < \infty$ (see page no 83 of \cite{TG}). 
If $\K^p[\mathbb{B}_j^n] $ is completion of $\mcL^p[\mathbb{B}_j^n],$ so the H\"{o}lder and generalized H\"{o}lder  inequalities hold in $\K^p[\mathbb{B}_j^n] $ for $ 1\leq p < \infty.$
\end{rem}
  \begin{thm}
   (The Minkowski Inequality) Let $1 \leq p < \infty$ and $f, g \in \K^p[\mathbb{B}_j^n].$ Then $f +g \in \K^p[\mathbb{B}_j^n] $ and $$||f + g||_{\K^p[\mathbb{B}_j^n]} \leq ||f||_{\K^p[\mathbb{B}_j^n]} +||g||_{\K^p[\mathbb{B}_j^n]}.$$
  \end{thm}
  \begin{proof}
   The proof follows from the Lemma 2 of \cite{LM}.
  \end{proof}
  \begin{thm}
  For $1\leq p \leq \infty,$ we have 
  \begin{enumerate}
  \item If $f_n \to f $ weakly in $\mcL^p[\mathbb{B}_j^n],$ then $f_n \to f$ strongly in $\K^p[\mathbb{B}_j^n].$
  \item If $ 1 < p < \infty,$ then $\K^p[\mathbb{B}_j^n]$ is uniformly convex.
  \item If $ 1 <p < \infty$ and $\frac{1}{p} + \frac{1}{q}=1,$ then the dual space of $\K^p[\mathbb{B}_j^n] $ is $\K^q[\mathbb{B}_j^n].$
  \item $\K^\infty[\mathbb{B}_j^n] \subset \K^p[\mathbb{B}_j^n],$ for $ 1\leq p < \infty.$
  \end{enumerate}
  \end{thm}
  \begin{proof}
  (1) If $ \{f_n\} $ is weakly convergence sequence in $ \mcL^p[\mathbb{B}_j^n] $ with limit $ f.$ Then $ \int_{\mathbb{B}_j^n}\mathcal{E}_k(x)[f_n(x)-f(x)]d\mu_{\la_\B}(x) \to 0 $ for each $k.$\\
 For each $ f_n \in \K^p[\mathbb{B}_j^n] $ for all n $$ \lim\limits_{n \to \infty}\int_{\mathbb{B}_j^n} \mathcal{E}_k(x)[f_n(x)-f(x)]d\mu_{\B}(x) \to 0.$$
 So, $\{f_n\} $ is converges strongly in $\K^p[\mathbb{B}_j^n]$.\\
 (2) 
 We know $\mcL^p[\mathbb{B}_j^n] $ is uniformly convex  and that is dense and compactly embedded in $\K^q[\mathbb{B}_j^n] $ for all $q,~~1\leq q \leq \infty. $\\
 So, $\bigcup\limits_{n=1}^{\infty}\mcL^p[\mathbb{B}_j^n]$ is uniformly convex for each $n $ and that is dense and compactly embedded in $\bigcup\limits_{n=1}^{\infty}\K^q[\mathbb{B}_j^n] $ for all $q,~1\leq q \leq \infty $. However $\mcL^p[\widehat{\mathbb{B}_j^n}]=\bigcup\limits_{n=1}^{\infty}\mcL^p[\mathbb{B}_j^n].$ 
 That is  $\mcL^p[\widehat{\mathbb{B}_j^n}] $ is uniformly convex, dense and compactly embedded in $\K^q[\widehat{\mathbb{B}_j^n}] $ for all $q,~~1\leq q \leq \infty $.\\
  as $\K^q[\mathbb{B}_j^n] $ is the closure of $\K^q[\widehat{\mathbb{B}_j^n}].$ 
  Therefore $\K^q[\mathbb{B}_j^n]$ is uniformly convex.\\
   $(3)$ We have from $(2)$, that $\K^p[\mathbb{B}_j^n]$ is reflexive for $ 1 < p < \infty.$ Since$$ \{\K^p[\mathbb{B}_j^n]\}^*= \K^q[\mathbb{B}_j^n],~\frac{1}{p} + \frac{1}{q}=1,~\forall n $$ and $$\K^p[\mathbb{B}_j^n] \subset \K^p[\mathbb{B}_J^{n+1}],~\forall n \implies \bigcup_{n=1}^{\infty}\{\K^p[\mathbb{B}_j^n]\}^* = \bigcup_{n=1}^{\infty}\K^q[\mathbb{B}_j^n],~\frac{1}{p}+\frac{1}{q}=1.$$ Since each $ f \in \K^p[\mathbb{B}_j^n] $ is the limit of a sequence $\{f_n\} \subset \K^p[\widehat{\mathbb{B}_j^n}]= \bigcup\limits_{n=1}^{\infty}\K^p[\mathbb{B}_j^n],$ we see that $\{\K^p[\mathbb{B}_j^n]\}^*= \K^q[\mathbb{B}_j^n], $ for $\frac{1}{p} + \frac{1}{q}=1.$\\
     (4) Let $ f \in \K^\infty[\mathbb{B}_j^n].$       This implies $|\int_{\mathbb{B}_j^n} \mcE_k(x)f(x) d\mu_{\B}(x)| $ is uniformly bounded for all $ k $. It follows that $|\int_{\mathbb{B}_j^n} \mcE_k(x)f(x) d\mu_{\B}(x)|^p $ is uniformly bounded for all $1\leq p < \infty.$      It is clear from the definition of $\K^p[\mathbb{B}_j^n] $     that \begin{align*}
     \left[\Sigma\left|\int_{\mathbb{B}_j^n} \mathcal{E}_k(x) f(x)d\mu_{\B}(x)\right|^p\right]^\frac{1}{p} &\leq M||f||_{\K^p[\mathbb{B}_j^n]}< \infty. 
\end{align*}   
So, $ f \in \K^p[\mathbb{B}_j^n].$  This completes the result.
  \end{proof}
   \begin{thm}$ C_{c}^{\infty}[\mathbb{B}_j^n] $ is a dense subset of  $ \K^2[\mathbb{B}_j^n].$
      \end{thm}
   \begin{proof}
    As $ C_{c}^{\infty}[\mathbb{B}_j^n] $ is dense in $ \mcL^p[\mathbb{B}_j^n], \forall p . $  Moreover $\mcL^p[\mathbb{B}_j^n]$ contained as dense subset of  $ \K^2[\mathbb{B}_j^n].$  So, $ C_{c}^{\infty}[\mathbb{B}_j^n] $ is a dense subset of  $ \K^2[\mathbb{B}_j^n].$
       \end{proof}
   \begin{cor}
   $C_{0}^{\infty}[\mathbb{B}_j^n] \subset \K^p[\mathbb{B}_j^n]$ as dense.
   \end{cor}
    
   \begin{rem}
    Since $\mcL^1[\mathbb{B}_j^n] \subset \K^p[\mathbb{B}_j^n] $ and $\K^p[\mathbb{B}_j^n]$ is reflexive for $ 1<p< \infty $. We see the second dual $ \{\mcL^1[\mathbb{B}_j^n]\}^{**} = \mfM[\mathbb{B}_j^n] \subset \K^p[\mathbb{B}_j^n], $ where $\mfM[\mathbb{B}_j^n] $ is the space of bounded finitely additive set functions defined on the Borel sets $\mfB[\mathbb{B}_j^n].$ 
   \end{rem}
   \subsection{The family of $\K^p[\mathbb{B}_J^\infty]$:}
    We can now construct the spaces $\K^p[\mathbb{B}_J^\infty],~1 \leq p \leq \infty,$ using the same approach that led to $\mcL^1[\mathbb{B}_J^\infty]. $ Since $\K^p[\mathbb{B}_J^\infty] \subset \K^p[\mathbb{B}_J^\infty],$ We define $\K^p[\widehat{\B}_j^\infty] = \bigcup_{n=1}^{\infty}\K^p[\mathbb{B}_j^n].$ 
   \begin{Def}
       We say that a measurable function $ f \in \K^p[\B_{j}^{\infty}],$ for $ 1 \leq p \leq \infty,$ if there is a Cauchy sequence $\{f_n\} \subset \K^p[\widehat{\B}_{j}^{\infty}] $ with $ f_n \in \K^p[\B_{j}^{n}] $ and $\lim\limits_{n \to \infty}f_n(x) = f(x)~ \mu_{\B}$-a.e.
       \end{Def}
       The functions in $\K^p[\widehat{\B}_{j}^{\infty}] $ differ from functions in its closure $\K^p[\mathbb{B}_J^\infty],$ by sets of measure zero.
       \begin{thm}
       $\K^p[\widehat{\B}_j^\infty]= \K^p[\mathbb{B}_J^\infty] $
       \end{thm}
       \begin{Def}\label{Def17}
     If $ f \in \K^p[\mathcal{B}_{j}^{\infty}],$ we define the integral of $ f$ by $$ \int_{\B_{I}^{\infty}} f(x)d\mu_{\B}(x) = \lim\limits_{n \to \infty}\int_{\B_{j}^{n}} f_n(x)d\mu_{\B}(x),$$ where $ f_n \in \K^p[\B_{j}^{n}]$  is any  Cauchy sequence convergerging to $f(x).$
     \end{Def}
     \begin{thm}
   If $ f \in \K^p[\B_{j}^{\infty}]$, then the integral of $ f$ defined in Definition \ref{Def17} exists and is unique for every $f \in \K^p[\B_{j}^{\infty}]$.
\end{thm}
\begin{proof} Since the family of functions $\{f_n\}$ is Cauchy, it is follows that if the integral exists, it is unique.  To prove existence, follow the standard argument and first assume that $f(x) \ge 0$.  In this case, the sequence can always be chosen to be increasing, so that the integral exists.  The general case now follows by the standard decomposition. 
\end{proof}
  \begin{thm}
  
  If $f \in \K^p[\mathbb{B}_J^\infty],$ then all theorems that are true for $f \in \K^p[\mathbb{B}_j^n],$ also hold for $ f \in \K^p[\mathbb{B}_J^\infty.]$
  \end{thm}
    \begin{rem}
    We can extend the Fourier transformation as well as Convolution transformation from $\K^p[\mathbb{B}_J^\infty] $ to $\K^2[\mathbb{B}_J^\infty].$  Finally we  also can show  $\K^2[\mathbb{B}_J^\infty]$ is a better choice of Hilbert space for Heisenberg theory.
    \end{rem}

             \begin{thm}
      $\K^p[\mathbb{B}_j^n] $ and $\K^p[\mathbb{B}_J]$ are equivalent spaces.
     \end{thm}
     \begin{proof}
     If $\mathbb{B}_j^n$ is a separable Banach space, $ T $ maps $\mathbb{B}_j^n$ onto $\mathbb{B}_J \subset \mathbb{B}_j^\infty,$ where $ T $ is a isometric isomorphism so that $\mathbb{B}_J$ is a embedding of $\mathbb{B}_j^n $ into $ R_I^\infty.$ This is how we able to define a Lebesgue integral on $\mathbb{B}_j^n $ using $\mathbb{B}_J$ and $T^{-1}.$ Thus $\K^p[\mathbb{B}_j^n] $ and $\K^p[\mathbb{B}_J]$ are not different space.
     \end{proof}
     \begin{thm}
     $\K^p[\mathbb{B}_J^\infty] \subset \K^p[\mathbb{B}_j^\infty]$ embedding as closed subspace.
     \end{thm}
     
     \begin{proof}
      as every separable Banach space can be embedded in $\mathbb{B}_j^\infty$ as a closed subspace containing $\mathbb{B}_J^\infty.$       So, $\K^p[\mathbb{B}_J^\infty] \subset \K^p[\mathbb{B}_j^\infty] $ embedding  as a closed subspace. That is $\K^p[\bigcup\limits_{n=1}^{\infty}\B_{J}^{n}] \subset \K^p[R_I^\infty] $ embedding as a closed subspace.\\
      So, $\K^p[\mathbb{B}_j^n] \subset \K^p[\mathbb{B}_j^\infty] $ embedding as a closed subspace. Finally we can conclude that $\K^p[\mathbb{B}_J^\infty] \subset \K^p[\mathbb{B}_j^\infty]$ embedding as closed subspace.
     \end{proof}
     \subsection{Fourier Transform and $\K^p[\mathbb{B}_{J}^\infty]$}
                In Chapter $2$ of \cite{TG} they define the Fourier transform as a mapping from a uniform convex Banach space to its dual space. This approach exploits the strong relationship between a uniform convex Banach space (see \cite{JA}) and a Hilbert space at the expense of a restricted Fourier transform (we refer \cite{LG}). They mention it is also possible to define the Fourier transform $\mathsf{f}$ as a mapping on $ \mcL^1[\mathbb{B}_j^\infty] $ to $ C_0[\mathbb{B}_j^\infty] $ for all $ n $.\\
                 As one fixed linear operator that extends to a definition on $ \mcL^1[\mathbb{B}_j^\infty] $. as in definition recalling $ J= [\frac{-1}{2}, \frac{1}{2}],~\overline{x}=(x_k)_{k=1}^{n},~~\widehat{x}=(x_k)_{k=n+1}^{\infty} $ and $h_n(\widehat{x})=\otimes_{k=n+1}^{\infty}\chi_J(x_k) .$ The measurable functions on $\mathbb{B}_j^\infty, M_{J}^{n} $ are defined by $f_n(x)= f_{n}^{n}(\overline{x})\otimes h_n(\widehat{x}),$ where $ f_{n}^{n}(\overline{x}) $ is measurable on $ \mathbb{B}_j^\infty, $ so $ M_{J}^{n} $ is a partial tensor product subspace generated by the unit vector $ h(x)= h_{0}(\widehat{x})$. From this, we see that all of the spaces of functions considered in Chapter $2$ of \cite{TG} are also partial tensor product spaces generated by $h(x).$\\
                 In this section we show how the replacement of $ \mcL^1[\mathbb{B}_j^\infty], C_{0}[\mathbb{B}_j^\infty]$ by  $ \K^p[\mathbb{B}_j^\infty](h), C_{0}[\mathbb{B}_j^\infty](h)$ allows us to offer a different approach to the Fourier transform.
                 \begin{Def} \cite{TG} $\mathfrak{f}(f_n)(x) $ mapping $ \mcL^1[\mathbb{B}_j^\infty](h) $ into $ C_0[\mathbb{B}_j^\infty](\widehat{h}) $ by $\mathfrak{f}(f_n)(x)= \otimes_{k=1}^{n}\mathfrak{f}_k(f_n)\otimes_{k=n+1}^{\infty}\widehat{h}_n(x),$ where the product of sine function $ \widehat{h}_{n}(\widehat{x}) = [ \otimes_{k=n+1}^{n}\sin\frac{\pi y_k}{\pi y_k}] $ in the Fourier transform of the product $\Pi_{k=n+1}^{\infty}J $ of the interval $ J $.
                 \end{Def}
                 \begin{thm} \cite{TG} The operator $\mathfrak{f}$ extends to a bounded linear mapping of $ \mcL^1[\mathbb{B}_j^\infty](h) $ into $ C_{0}[\mathbb{B}_j^\infty](\widehat{h}).$
                 \end{thm}
                 Based on these results, we get the following result.
                 \begin{thm} The operator $\mathfrak{f}$ extends to a bounded linear mapping of $ \K^p[\mathbb{B}_j^\infty](h) $ into $ C_{0}[\mathbb{B}_j^\infty](\widehat{h}).$
                 \end{thm}
                 \begin{proof}
                  Since $$\lim\limits_{n \to \infty}\K^p[\mathbb{B}_j^\infty](h) = \bigcup_{n=1}^{\infty}\K^p[\mathbb{B}_j^\infty](h) = \K^p[\widehat{\mathbb{B}}_{J}^{\infty}].$$ also $ \K^p[\mathbb{B}_j^\infty](h) $ is the closure of $\K^p[\widehat{\mathbb{B}}_{J}^{\infty}]$ in $\K^p$-norm  it follows that $\mathfrak{f}$ is bounded linear mapping of $\K^p[\widehat{\mathbb{B}}_{J}^{\infty}]$ onto $C_0[\mathbb{B}_j^\infty](\widehat{h}).$\\
                  Suppose the sequence $\{f_n\} \subset \K^p[\widehat{\mathbb{B}}_{J}^{\infty}] $ converges to $f \in \K^p[\mathbb{B}_j^\infty](h),$ since the sequence is Cauchy, so  $||f_n -f_m||_p \to 0 $ as $n,m \to \infty,$ it follows that 
                  \begin{align*}
                  |\mathfrak{f}(f_n(x))-f_m(x))| &\leq \int_{\mathbb{B}_j^\infty}|f_n(y)-f_m(y)|d\mu_{\mathbb{B}}(y) \\&=||f_n-f_m||_1.
                  \end{align*}
                  Thus $|\mathfrak{f}(f_n(x))-f_m(x))|$ is a Cauchy sequence in $C_0[\mathbb{B}_j^\infty](\widehat{h})$. Since  $\K^p[\widehat{\mathbb{B}}_{J}^{\infty}]$ is dense in $ \K^p[\mathbb{B}_j^\infty](h) $, so $\mathfrak{f}$ has a bounded extension mapping  from $ \K^p[\mathbb{B}_j^\infty](h) $ into $ C_{0}[\mathbb{B}_j^\infty](\widehat{h}).$
                 \end{proof}
       \subsection{Feynman Path Integral}
       The properties of $\K^2[\mathbb{B}_J^\infty]$ derived earlier suggests that it may be a better replacement of $\mcL^2[\mathbb{B}_J^\infty] $ in the study of the Path Integral formulation of quantum theory developed by Feynman. 
              We see that  position operator have closed densely defined extensions to $\K^2[\mathbb{B}_J^\infty].$ Further Fourier and convolution insure that all of the Schr\"{o}dinger and Heisenberg theories have a faithful representation on $\K^2[\mathbb{B}_J^\infty].$ Since $\K^2[\mathbb{B}_J^\infty] $ contain the space of measures, it follows that all the approximating sequences for Dirac measure convergent strongly in $\K^2[\mathbb{B}_J^\infty].$

  \bibliographystyle{amsalpha}

\begin{thebibliography}{a}
	
	
	\bibitem {AL} A.  Alexiewicz,    Linear functionals on Denjoy-integrable functions,  Colloq. Math.  1(1948) 289--293.
	\bibitem{MA} M.A. Algwaiz, \textit{ Theory of Distributions,} Pure and Applied Mathematics, Monograph, Marcel Dekker, Inc. (1992).
	
	\bibitem{ASV} D.D. Ang, K. Schmitt, L.K. Vy, A multidimensional analogue
	of the Denjoy-Perron-Henstock-Kurzweil integral. Bull. Belg.
	Math. Soc. Simon Stevin 4(1997) 355–371.
	
	
	\bibitem{BR} R.G. Bartle, Return to the Riemann integral. Amer. Math. Mon.
	103(1996) 625–632.
	\bibitem{BO1} S. Bochner, Additive set functions on groups. Ann. Math. 40(1939) 769–799. 
	\bibitem{BO2} S. Bochner, Finitely additive integral. Ann. Math. 41(1940) 495–504.
	
	\bibitem{JA} J.A. Clarkson, Uniformly Convex Spaces, Trans. Amer. Math. Soc. 40(3)(1936) 396--414.
	
	\bibitem{DZ} R.O. Davies, Z. Schuss, A proof that Henstock’s integral includes Lebesgue’s, Lond. Math. Soc. 2(1970) 561–562.
	
	
	\bibitem{TG} T. L. Gill, W. W. Zachary, \textit{  Functional Analysis and The  Feynman Operator Calculus,} Springer International Publishing Switzerland, (2016). 
	
	\bibitem{GM} T. L. Gill, T. Myers, Constructive Analysis on Banach spaces, Real Analysis Exchange
	44(1)(2019) 1--36.
	\bibitem{RA}  R. A. Gordon, \textit{The Integrals of Lebesgue, Denjoy, Perron and Henstock}, Graduate Studies in Mathematics Vol 4, American Mathematical Society, (1994).
	
	\bibitem{LG} L. Grafakos, \textit{ Modern Fourier Analysis,} Springer (2014).
	
	
	
	
	
	\bibitem{GR} L. Gross, Abstract Wiener spaces, in Proceedings of 5th
	Berkeley Symposium on Mathematics Statistics and Probability (1965) pp 31–42.
	\bibitem{HS} R. Henstock, The General Theory of Integration, Clarendon
	Press, Oxford (1991).
	
	\bibitem{RH}  R. Henstock, \textit{ The General Theory of Integration,} Oxford mathematical monographs, Clarendo Press, Oxford (1991). 
	
	
	\bibitem{HT}R. Henstock, Lectures on the Theory of Integration. World Scientific, Singapore, (1988). 
	\bibitem{H}R. Henstock, A Riemann-type integral of Lebesgue power, Canadian J. Math. 20 (1968) 79-87.
	\bibitem{HI} T.H. Hildebrandt, On bounded functional operations. Trans. Am. Math. Soc. 36(1934) 868–875.
	
	\bibitem{KB} J. Kuelbs, Gaussian measures on a Banach space. J. Funct. Anal. 5(1970) 354--367.
	
	
	
	\bibitem{KZ}J. Kurzweil, Henstock-Kurzweil Integration: Its Relation to Topological Vector Spaces. World Scientific, Singapore, (2000).
	
	\bibitem{KW} J. Kurzweil, Nichtabsolut konvergente Integrale. Teubner-Texte
	zür Mathematik, Band 26, Teubner Verlagsgesellschaft, Leipzig, (1980).
	
	\bibitem{L} P.D. Lax, Symmetrizable linear transformations. Commun. Pure Appl. Math. 7(1954) 633–647.
	\bibitem{LE} S. Leader, The theory of $L^p$-spaces for finitely additive set functions. Ann. Math. 58(1953) 528–543.
	\bibitem{LMB} L. Lorenzi, M. Bertoldi, \textit{ Analytical Methods for Markov semi groups},  Chapman and Hall/CRC Press (2006).
	
	\bibitem{AR}  A.R. Lovaglia, Locally Uniformly Convex Banach Space, Trans. Amer. Math. Soc. 78(1)(1955) 225--238.
	
	
	
	
	
	
	
	
	
	
	
	\bibitem{LM} L. Maligranda, A simple proof of the H\"{o}lder and Minkowski Inequality, The American Mathematical Monthly, 102(3)(1995) 256--259.
	
	
	\bibitem{RTM} R. Rudnicki, T. Pichor, M.T. Kaminska, Markov Semi groups and Their application, Dynamics of Dissipation, P. Garbaczewski and R. Olhiwicz lecture note in Physics, vol 597, Springer, Berlin, 215--238.
	\bibitem{CS} C. Schwartz,\textit{ Introduction to Gauge Integral,} World Scientific Publisher, (2004).
	\bibitem{SM} V. Steadman, Theory of operators on Banach spaces. PhD thesis, Howard University, 1988
	\bibitem{SY} S. Schwabik,  Ye Guoju, \textit{Topics in Banach Spaces Integration, Series in Real Analysis}, Vol 10. World scientific, Singapore (2005). 
	
	\bibitem{BST}  B. S. Thomson, \textit{ Theory of Integral,} Classical Real Analysis.com (2013).
	
	
	\bibitem{YY} Y. Yamasaki, Measures on infinite dimensional spaces, Series in Pure Mathematics, Volume 5, World Scientific Publishing Co. Pvt. Ltd (1985).
	\bibitem{TY} L. Tuo-Yeong, \textit{Henstock-Kurweil Integration on Euclidean Spaces,} Series in Real Analysis, vol 12, World Scientific, New Jersey (2011).
	
	
	
	
	
	
	
	
	































 


 \end{thebibliography}

\end{document}